\documentclass[11pt]{amsart}
\usepackage[margin=1.0in]{geometry}
\usepackage{amsmath,amssymb,amsfonts,graphicx,color,fancyhdr,psfrag,comment,enumerate}
\usepackage[latin1]{inputenc}
\usepackage[hyperpageref]{backref}
\usepackage[colorlinks=true, pdfstartview=FitV, linkcolor=blue, citecolor=blue, urlcolor=blue]{hyperref}

\usepackage{epstopdf}
\usepackage{mathrsfs}
\usepackage{epsfig}
\usepackage{hyperref}
\usepackage{ifthen}
\usepackage{float}
\usepackage{enumerate}
\usepackage{mathtools}
\usepackage{tikz}
\usetikzlibrary{matrix,shapes,arrows,positioning,chains}
\usepackage{amssymb,amsmath,amsfonts}
\usepackage{bbm}

\numberwithin{equation}{section}

\newtheorem{theorem}{Theorem}[section]

\newtheorem{lemma}[theorem]{Lemma}
\newtheorem{definition}[theorem]{Definition}

\newtheorem{remark}[theorem]{Remark}

\newcommand{\HL}{Hardy-Littlewood}

\newcommand{\cT}{\mathcal{T}^k}
\newcommand{\cTi}{\mathcal{T}_{i}}
\newcommand{\cTj}{\mathcal{T}_{j}}

\def\XXint#1#2#3{{\setbox0=\hbox{$#1{#2#3}{\int}$}
    \vcenter{\hbox{$#2#3$}}\kern-.5\wd0}}

\newcommand{\vertiii}[1]{{\left\vert\kern-0.25ex\left\vert\kern-0.25ex\left\vert #1 
    \right\vert\kern-0.25ex\right\vert\kern-0.25ex\right\vert}}

\numberwithin{equation}{section}

\allowdisplaybreaks

\begin{document}

\title[Fractional maximal function on the infinite rooted $k$-ary tree]{Weighted Inequalities for Fractional maximal functions on the infinite rooted $k$-ary tree}
\author[A. Ghosh, E. Rela]{Abhishek Ghosh \and Ezequiel Rela}

\address[A. Ghosh]{Tata Institute of Fundamental Research, Centre for Applicable Mathematics, Bangalore--560065, Karnataka, India.}
\email{abhi170791@gmail.com, abhi21@tifrbng.res.in}

\address[E. Rela]{Department of Mathematics,
University of Buenos Aires, Ciudad Universitaria
Pabell\'on I, Buenos Aires (1428), Argentina} \email{erela@dm.uba.ar}

\thanks{E.R. is partially supported by grants UBACyT 20020170200057BA and PIP (CONICET) 11220110101018}



\begin{abstract}
In this article we introduce the fractional \HL~maximal function on the infinite rooted $k$-ary tree and study its weighted boundedness. We also provide examples of weights for which the fractional \HL~maximal function satisfies strong type $(p, q)$ estimates on the infinite rooted $k$-ary tree.
\end{abstract}

\maketitle

\section{Introduction and Preliminaries}
Weighted estimates for classical operators in Harmonic analysis has always been an active and growing area of research. The seminal work of Muckenhoupt \cite{Muckenhoupt:Ap} introduced the $A_p$ class of weights and characterized the weighted inequalities for the \HL~maximal function. Subsequently, Muckenhoupt and Wheeden in \cite{MW-fractional} studied weighted inequalities for the fractional \HL~maximal function and introduced the $A_{p, q}$ class of weights. In this article we initiate the study of weighted estimates for the fractional \HL~maximal function on infinite graphs. In recent times there are several substantial works devoted to the study of the discrete analogues of the \HL~maximal function and the fractional \HL~maximal function in various frameworks, for instance we refer \cite{Carneiro-Hughes, JoseMAdrid-1, JoseMadrid-2}. Our main motivation for studying this object originates from the recent works \cite{Ombrosi-Rios-Safe} and \cite{Ombrosi-Rios}. In \cite{Ombrosi-Rios-Safe}, Ombrosi, Rivera-R\'ios and Safe have proved a sharp analog of Fefferman-Stein inequality for the \HL~maximal operator on the infinite rooted $k$-ary tree and subsequently in \cite{Ombrosi-Rios} weighted inequalities for the \HL~maximal function were investigated by Ombrosi and Rivera-R\'ios. We also refer the articles \cite{Soria-Tradacete-JMAA, Soria-Tradacete} where the authors studied connections between geometrical properties of infinite graphs and the boundedness of the \HL~ maximal operator. To elaborate the results let us start with preliminaries.

For $k\geq 2$, we denote $\mathcal{T}^{k}$ to be the infinite rooted $k$-ary tree i.e., the infinite tree where each vertex has $k$ successors. $(\mathcal{T}^{k}, d, |.|)$ is a metric measure space where $d$ is the standard tree distance, i.e., $d(x, y)$ is the number of edges in the minimal unique path joining $x$ and $y$ and $|A|$ denotes the counting measure of any set $A\subset \cT$. Given any function $g$, $\int_{A}g(y)\, dy:=\sum_{y\in A}g(y)$ and $B(x, r):=\{y\in \cT: d(x, y)\leq r\}$ will denote the ball of radius $r$ centred at the point $x\in \cT$. Set $\mathbb{N}_{0}=\mathbb{N}\cup \{0\}$. Now let us define the fractional maximal operator in this setting.
\begin{definition}
Fix $0\leq \alpha<1$. Let $f$ be any locally integrable function then the fractional \HL~maximal operator is defined as follows
\begin{align*}
M_{\alpha}f(x)=\sup_{r\in \mathbb{N}_{0}}\frac{1}{|B(x, r)|^{1-\alpha}}\int_{B(x, r)}|f(y)|\, dy.
\end{align*}
\end{definition}
\medskip
 For $\alpha=0$, $M_{0}$ is the classical \HL~maximal function and will be simply denoted by $M$. Let $S(x, r):=\{y: d(x, y)=r\}$ denote the sphere, centred at $x$, of radius $r\in \mathbb{N}_{0}$. We also define the spherical fractional maximal operator as following
\begin{align*}
\mathcal{S}_{\alpha}f(x)=\sup_{r\in \mathbb{N}_{0}}\frac{1}{|S(x, r)|^{1-\alpha}}\int_{S(x, r)}|f(y)|\, dy.
\end{align*}
We will denote by $A_{r, \alpha}f(x):=\frac{1}{|S(x, r)|^{1-\alpha}}\int_{S(x, r)}|f|$ to be the spherical fractional average of $f$ at $x$ over the sphere $S(x, r)$. Unlike the Euclidean setting, it is easy to see that the operators $M_{\alpha}$ and $\mathcal{S}_{\alpha}$ are pointwise equivalent i.e., $C_{\alpha} \mathcal{S}_{\alpha} f(x)\leq M_{\alpha}f(x)\leq c_{\alpha} \mathcal{S}_{\alpha} f(x)$ for a.e. $x\in \mathcal{T}^k$. This is due to the fact that $|S(x, r)|\simeq k^r$ and $\frac{|B(x, r)|}{|S(x, r)|}\leq 2$ (see \cite{Ombrosi-Rios-Safe}). Keeping this in mind we will only write our results for the spherical fractional  maximal function  which can be easily transferred to the fractional \HL~maximal function.

Let us review some unweighted boundedness for the  spherical fractional maximal operator $\mathcal{S}_{\alpha}$. Note that $\mathcal{S}_{\alpha}f(x)\lesssim \|f\|_{L^1(\mathcal{T}^k)}^{\alpha} (Mf(x))^{1-\alpha}$ which implies that $\mathcal{S}_{\alpha}$ maps $L^1(\mathcal{T}^k)$ to $L^{\frac{1}{1-\alpha}, \infty}(\mathcal{T}^k)$ boundedly, where we have used the fact that $M$ is of weak type $(1, 1)$. H\"older's inequality implies that $\mathcal{S}_{\alpha}: L^{1/\alpha}(\mathcal{T}^k)\to L^{\infty}(\mathcal{T}^k)$ boundedly. Interpolating the above two, we obtain that $\mathcal{S}_{\alpha}: L^p(\mathcal{T}^k)\to L^q(\mathcal{T}^k)$, where $\frac{1}{q}=\frac{1}{p}-\alpha$ and $1<p<\frac{1}{\alpha}$. In this article we will provide sufficient conditions on weights $w$ for which $\mathcal{S}_{\alpha}$ maps $L^p(w)$ to $L^q(w)$ with $1<p\leq q<\infty$. 

We would like to point out that in \cite{Ombrosi-Rios} the authors showed that there exist weights for which the \HL~maximal function is bounded from $L^p(\cT)$ to itself but the obvious Muckenhoupt-type condition does not hold. This suggests that the classical $A_p$-theory is not applicable for the infinite rooted $k$-ary tree. To overcome this difficulty, let us start with obtaining some useful necessary conditions for the boundedness of the spherical fractional \HL~maximal function on the infinite rooted $k$-ary tree.

For any set $A\subset \cT$ and any weight $w$, denote $w(A):=\sum_{x\in A}w(x)$. It is easy to see that if $\mathcal{S}_{\alpha}$ maps $L^{p}(w)$ to $L^{q}(w)$ then for any two finite subsets $E$ and $F$ of $\mathcal{T}^k$ and any $r\in \mathbb{N}_0$ we have
\begin{align}
\label{necessary-cond}
\nonumber \sum_{x\in E}\frac{w(F\cap S(x, r))}{|S(x, r)|^{1-\alpha}}&=\sum_{x\in E}A_{r, \alpha}(w \chi_{F})(x)\\
\nonumber&=\int_{\mathcal{T}^k}\chi_{E}(x)A_{r, \alpha}(w \chi_{F})(x)\\
\nonumber &=\int_{\mathcal{T}^k}A_{r, \alpha}(\chi_{E})(x)w(x) \chi_{F}(x)\\
\nonumber &\leq \|A_{r, \alpha}\chi_{E}\|_{L^q(w)}\|\chi_{F}\|_{L^{q'}(w)}\\
\nonumber &\lesssim \|\chi_{E}\|_{L^p(w)}\|\chi_{F}\|_{L^{q'}(w)}\lesssim w(E)^{\frac{1}{p}}w(F)^{1-\frac{1}{q}},
\end{align}
where we have used the self-adjointness of the operator $A_{r, \alpha}$. Hence the condition 
\begin{equation}
\label{Nec}
\sum_{x\in E}w(F\cap S(x, r))\lesssim k^{r(1-\alpha)}w(E)^{\frac{1}{p}}w(F)^{1-\frac{1}{q}}
\end{equation}
is a necessary condition for the boundedness of $\mathcal{S}_{\alpha}$ from $L^{p}(w)$ to $L^{q}(w)$. However, the condition \eqref{Nec} will not be sufficient for the boundedness of $\mathcal{S}_{\alpha}$, but assuming little better decay on $k$ we are able to obtain the boundedness. Motivated by this we define the following class of weights.
\begin{definition}
Let $1<p, q<\infty$ and $0<\alpha<1$. We say $w\in \mathcal{Z}_{p, q}^{\epsilon, \alpha}$ if there exist $\epsilon\in  (0, 1)$, and a constant $C>0$ such that for all $r\in \mathbb{N}$ and all measurable sets $E, F\subset \cT$ we have
\begin{align}
\label{sufficient}
\sum_{x\in E}w(F\cap S(x, r))\leq C\, k^{\epsilon r(1-\alpha)} w(E)^{\frac{1}{p}}w(F)^{1-\frac{1}{q}}.   
\end{align}
\end{definition}
Now we are at a position to state the main result of this article. 
\begin{theorem}
\label{mainthm}
Let $1<p\leq q<\infty$ and $0<\alpha<1$. Then for $w\in \mathcal{Z}_{p, q}^{\epsilon, \alpha}$\, we have $\mathcal{S}_{\alpha}: L^p(w)\to L^q(w)$ boundedly.  
\end{theorem}
We note here that the exponents $p,q$ in the above theorem are not related to the parameter $\alpha$ and this might look strange, compared to the unweighted theory. The key is that in order to have a nontrivial example of a weight belonging to the class $\mathcal{Z}_{p, q}^{\epsilon, \alpha}$ satisfying \eqref{sufficient} we will need to consider the classical condition $\frac{1}{q}=\frac{1}{p}-\alpha$ (see Theorem \ref{Example-Thm} and Theorem \ref{Thm-Ex}).

Since the volume of the balls grows exponentially in the infinite rooted $k$-ary tree, i.e., $\frac{|B(x, r)|}{|B(x, r/2)|}\simeq k^{r/2}\to \infty$ as $r\to \infty$, the space $(\mathcal{T}^k, d, |.|)$ is not even an upper-doubling space. Therefore, covering arguments and Calder\'on-Zygmund decomposition are not well suited here. However, overcoming these obstacles, the weak type $(1, 1)$ estimate for the \HL~maximal function on infinite rooted $k$-ary tree was obtained in \cite{Naor-Tao}. In \cite{Naor-Tao}, Naor and Tao developed an ingenious strategy, based on ``expander" properties of graphs and combinatorial arguments, to obtain the weak type $(1, 1)$ boundedness of the \HL~maximal function with bounds independent of $k$ (see Theorem~1.5 in \cite{Naor-Tao}). We would also like to note that, in \cite{Cowling-Meda-Setti}, the same result was proved using a different approach. The idea of our proofs is inspired by the works \cite{Ombrosi-Rios-Safe} and \cite{Ombrosi-Rios} where they have brilliantly implemented the techniques developed by Naor and Tao in the weighted setting. We also refer the article \cite{Soria-Tradacete-JMAA} where the authors extend the result of Naor and Tao to more general infinite graphs. We end this section by constructing examples of appropriate weights.

Let us introduce some notations. Set $\mathcal{T}_{0}$ as the set containing the root of the infinite tree and $\mathcal{T}_{i}$ denotes the set of children of the vertices in $\mathcal{T}_{i-1}$ for all $i\geq 1$. Then $\mathcal{T}^{k}=\bigcup\limits_{i=0}^{\infty}\mathcal{T}_{i}$. The examples of weights belonging to the class $\mathcal{Z}_{p, q}^{\epsilon, \alpha}$ are not easy to obtain since the condition \eqref{sufficient} is generally difficult to verify. One of the primary reason being the fact that the condition \eqref{sufficient} involves arbitrary sets and their possible interactions with multiple levels $\cTi$. To overcome this difficulty, let us first provide a simpler criterion which only depends on the behaviour of the weight on individual scales $\cTi$. In that direction our result reads as follows. 
\begin{theorem}
\label{Example-Thm}
Let $\frac{1}{q}=\frac{1}{p}-\alpha$ with $1<p<\frac{1}{\alpha}$ and $0<\alpha<1$. Denote $\delta=\frac{1-\alpha p-\alpha p^2}{1-\alpha p}$ and let $w$ be a weight satisfying the following condition: For $r\in \mathbb{N}_{0}$ and $|i-j|\leq r$, we have that
\begin{equation}
\label{cond-weight-1}
w(\mathcal{T}_{i}\cap S(x, r))\lesssim k^{\frac{r+i-j}{2}(p-\delta)} k^{r\delta} w(x)^{q/p}~~\text{for~all}~~x\in \cTj. 
\end{equation}
Then $w\in \mathcal{Z}^{\epsilon, \alpha}_{p, q}$ for $\epsilon=\frac{1-p\alpha}{1-\alpha}$.
\end{theorem}
The proof of the above theorem relies on an optimization technique developed in \cite{Naor-Tao, Ombrosi-Rios}. However, we modify it significantly since in our case we are dealing with the much more general situation of off-diagonal estimates for the fractional maximal function. Finally the following result provides some concrete example of weights.
\begin{theorem}
\label{Thm-Ex}
Let $0<\alpha<1$, $1<p<\frac{1}{\alpha}$ and $\frac{1}{q}=\frac{1}{p}-\alpha$. Define $$w(x)=\sum_{j\geq 0} k^{\beta j}\chi_{\cTj},$$
where $0\leq \beta\leq \frac{p(p-1)}{q}$. Then $w$ belongs to $\mathcal{Z}^{\epsilon, \alpha}_{p, q}$ for $\epsilon=\frac{1-p\alpha}{1-\alpha}$  and hence $\mathcal{S}_{\alpha}$ maps $L^p (w)$ to $L^q(w)$ boundedly.
\end{theorem}
The article is organized as follows. In the next section we prove the main boundedness result i.e., Theorem~\ref{mainthm}. Section~\ref{examples} is devoted in proving  Theorem~\ref{Example-Thm} and Theorem~\ref{Thm-Ex}. Throughout this article $A\lesssim B$ abbreviates that $A\leq C B$ where the constant $C$ is independent of $A$ and $B$. We will also be using the inequality $\sum_{j=0}^{\infty}a_{j}^\theta\leq (\sum_{j=0}^{\infty} a_{j})^{\theta}$, where $a_{j}\geq 0$ and $1\leq \theta<\infty$, frequently.

\section{Main results}
Let us prove the following lemma. 
\begin{lemma}
\label{main-lemma}
Let $1<p\leq q<\infty$, $0<\alpha<1$ and $w\in \mathcal{Z}_{p, q}^{\epsilon, \alpha}$ for some $\epsilon\in (0, 1)$. Then for $\lambda>0, r\in \mathbb{N}$, we have the following estimate
\begin{equation*}
w(A_{r, \alpha}f>\lambda)\lesssim \frac{1}{(2^\beta-1)^q} \sum\limits_{n\in\mathbb{N}_{0}:\, 1\leq 2^n\leq k^r}2^{n q}\left(\frac{k^r}{2^n}\right)^{q \beta} \frac{k^{rq\epsilon(1-\alpha)}}{k^{rq}}w\left(\Big\{x\in \cT:\frac{|f| k^{r\alpha} }{\lambda}\geq 2^{n-1}\Big\}\right)^{\frac{q}{p}}
\end{equation*}
for all $\beta\in (0, 1)$.
\end{lemma}

\begin{proof}
Without loss of generality let us assume $f$ to be a non-negative function. We can write
\begin{align*}
    f\leq \frac{1}{2}+\sum\limits_{n\in\mathbb{N}_{0}: 1\leq 2^n\leq k^r}2^n \chi_{E_{n}}+f \chi_{\{f\geq \frac{1}{2}k^{r\alpha}\}},
\end{align*}
where $E_{n}=\{2^{n-1}\leq f<2^n\}$. This implies
\begin{align*}
    A_{r, \alpha}f(x)\leq \frac{k^{r\alpha}}{2}+\sum\limits_{n\in\mathbb{N}_{0}: 1\leq 2^n\leq k^r}2^n A_{r, \alpha}(\chi_{E_{n}})(x)+ A_{r, \alpha}(f \chi_{\{f\geq \frac{1}{2}k^{r\alpha}\}})(x).
\end{align*}
Consider the set $E=\{x: A_{r, \alpha}f(x)>k^{r\alpha}\}$. Using the above estimate, we obtain the following inequality for all $x\in E$
\begin{equation*}
 \sum\limits_{n\in\mathbb{N}_{0}: 1\leq 2^n\leq k^r}2^n A_{r, \alpha}(\chi_{E_{n}})(x)+ A_{r, \alpha}(f \chi_{\{f\geq \frac{1}{2}k^{r\alpha}\}})(x)\geq \frac{1}{2}k^{r\alpha},
 \end{equation*}
 which implies that 
\begin{equation*}
\sum\limits_{n\in\mathbb{N}_{0}: 1\leq 2^n\leq k^r}2^n A_{r, \alpha}\left(\frac{1}{k^{r\alpha}}\chi_{E_{n}}\right)(x)+ A_{r, \alpha}\left(\frac{1}{k^{r\alpha}} f \chi_{\{f\geq \frac{1}{2}k^{r\alpha}\}}\right)(x)\geq \frac{1}{2}.
\end{equation*}
Hence $w(E)\leq w(I)+w(II)$, where
\begin{equation*}
I:=\Bigg\{x: \sum\limits_{n\in\mathbb{N}_{0}: 1\leq 2^n\leq k^r}2^n A_{r, \alpha}\left(\frac{1}{k^{r\alpha}}\chi_{E_{n}}\right)(x)\geq \frac{1}{4}\Bigg\},
\end{equation*}
and
\begin{equation*}
II:= \Big\{x: A_{r, \alpha}\left(\frac{1}{k^{r\alpha}} f \chi_{\{f\geq \frac{1}{2}k^{r\alpha}\}}\right)(x)\geq 0\Big\}.
\end{equation*}
Let us fix some $\beta\in (0, 1)$. For $x\in I$, there exist some $n\in \mathbb{N}_{0}$ with $1\leq 2^n\leq k^r$ such that 
\begin{equation*}
2^n A_{r, \alpha}\left(\frac{1}{k^{r\alpha}}\chi_{E_{n}}\right)(x)\geq \frac{1}{16}(2^\beta-1)\left(\frac{2^n}{k^{r}}\right)^{\beta}.
\end{equation*}
Otherwise if we have  
$2^n A_{r, \alpha}\left(\frac{1}{k^{r\alpha}}\chi_{E_{n}}\right)(x)< \frac{1}{16}(2^\beta-1)\left(\frac{2^n}{k^{r}}\right)^{\beta}$, for all $n\in \mathbb{N}_{0}$ with $1\leq 2^n\leq k^r$,\, we will have a contradiction.

Let us define 
\begin{equation*}
F_{n}:=\Bigg\{x: A_{r, \alpha}(\chi_{E_{n}})(x)\geq \frac{1}{2^{n+4}}(2^\beta-1)k^{r\alpha}\left(\frac{2^n}{k^{r}}\right)^{\beta}\Bigg\}.
\end{equation*}
Thus to estimate $w(I)$, it is enough to estimate  $\sum\limits_{n\in\mathbb{N}_{0}: 1\leq 2^n\leq k^r}w(F_{n})$. Let us estimate each $w(F_n)$. Using the self-adjointness of the operator $A_{r, \alpha}$ and since $F_n$ is finite we obtain
\begin{eqnarray*}
\sum_{x\in E_{n}}\frac{w(F_{n}\cap S(x, r))}{|S(x, r)|^{1-\alpha}}& = & \sum_{x\in E_{n}} A_{r, \alpha}(w \chi_{F_n})(x)\\
&=&\int_{\cT}\chi_{E_{n}} A_{r, \alpha}(w \chi_{F_n})\\
& = & \int_{\cT} A_{r, \alpha}(\chi_{E_{n}})w \chi_{F_n}\\
& \geq &  \frac{1}{2^{n+4}}(2^\beta-1)k^{r\alpha}\left(\frac{2^n}{k^{r}}\right)^{\beta} w(F_n).
\end{eqnarray*}
Recall that the condition on the weight is as follows: \begin{equation}
\label{condition2}\sum_{x\in E}w(F\cap S(x, r))\lesssim k^{r\epsilon(1-\alpha)}w(F)^{1/q'}w(E)^{1/p}.
\end{equation}
From the above two estimates we obtain
\begin{equation*}
\frac{1}{2^{n+4}}(2^\beta-1)k^{r\alpha}\left(\frac{2^n}{k^{r}}\right)^{\beta} w(F_n)\lesssim \frac{k^{r\epsilon(1-\alpha)}}{k^{r(1-\alpha)}}w(F_n)^{1/q'}w(E_n)^{1/p},
\end{equation*}
which is equivalent to 
\begin{equation*}
w(F_n)^{1/q}\lesssim \frac{1}{2^\beta-1}2^n\left(\frac{k^r}{2^n}\right)^{\beta}\frac{k^{r\epsilon(1-\alpha)}}{k^r} w(E_{n})^{1/p}.
\end{equation*}


We then obtain that
\begin{equation}\label{est0}
 w(F_n)\lesssim \frac{1}{(2^\beta-1)^q}2^{n q}\left(\frac{k^r}{2^n}\right)^{q \beta} \frac{k^{rq\epsilon(1-\alpha)}}{k^{rq}}w\left (\{f\geq 2^{n-1}\}\right)^{q/p}.
\end{equation}
Thus we have the following estimate
\begin{equation}
\label{est1}
w(I)\leq \sum\limits_{n\in\mathbb{N}_{0}: 1\leq 2^n\leq k^r}\frac{1}{(2^\beta-1)^q}2^{n q}\left(\frac{k^r}{2^n}\right)^{q \beta} \frac{k^{rq\epsilon(1-\alpha)}}{k^{rq}}w\left(\{f\geq 2^{n-1}\}\right)^{q/p}.
\end{equation}
Now let us estimate $w(II)$.
\begin{align}
\nonumber w(II)&:=w\left(\Big\{x: A_{r, \alpha}\left(\frac{1}{k^{r\alpha}} f \chi_{\{f\geq \frac{1}{2}k^{r\alpha}\}}\right)(x)\geq 0\Big\}\right)\\
\nonumber &=w\left(\Big\{x: A_{r, \alpha}\left( f \chi_{\{f\geq \frac{1}{2}k^{r\alpha}\}}\right)(x)\geq 0\Big\}\right)\\
\nonumber &\leq w\left(\bigcup_{y\in \{f\geq \frac{1}{2}k^{r\alpha}\}} S(y, r)\right) \\
& \leq \sum_{y\in \{f\geq \frac{1}{2}k^{r\alpha}\}}w(S(y, r)).
\label{est2}
\end{align}
If we take $E=\{y\}$ and $F=S(y,r)$ in \eqref{condition2}, we obtain
\begin{equation*}
w(S(y, r))\leq k^{r\epsilon(1-\alpha)}w(S(y, r))^{1/q'}w(y)^{1/p},
\end{equation*}
yielding
\begin{equation*}
 w(S(y, r))\leq k^{r\epsilon(1-\alpha)q} w(y)^{q/p}.  
\end{equation*}

Plugging the above estimate in \eqref{est2} we obtain
\begin{align*}
w(II)\leq \sum_{y\in \{f\geq \frac{1}{2}k^{r\alpha}\}} k^{r\epsilon(1-\alpha)q} w(y)^{q/p}\leq k^{r\epsilon(1-\alpha)q} w\left(\Big\{f\geq \frac{1}{2}k^{r\alpha}\Big\}\right)^{q/p},
\end{align*}
where we have used the fact that $q\geq p$ in the last inequality. Observe that the estimate for $w(II)$ is same as the final term in the summation in \eqref{est1}. Combining the above estimates we obtain
\begin{align*}
w(\{x: A_{r, \alpha}f>k^{r\alpha}\})\leq \frac{1}{(2^\beta-1)^q} \sum\limits_{n\in\mathbb{N}_{0}: 1\leq 2^n\leq k^r}2^{n q}\left(\frac{k^r}{2^n}\right)^{q \beta} \frac{k^{rq\epsilon(1-\alpha)}}{k^{rq}}w(\{f\geq 2^{n-1}\})^{q/p}.
\end{align*}
Using the above estimate and homogeneity we obtain
\begin{equation*}
w(\{A_{r, \alpha}f>\lambda\})\leq \frac{1}{(2^\beta-1)^q} \sum\limits_{n\in\mathbb{N}_{0}: 1\leq 2^n\leq k^r}2^{n q}\left(\frac{k^r}{2^n}\right)^{q \beta} \frac{k^{rq\epsilon(1-\alpha)}}{k^{rq}}w\left(\Big\{x\in \cT:\frac{f k^{r\alpha} }{\lambda}\geq 2^{n-1}\Big\}\right)^{q/p}.
\end{equation*}
This completes the proof of the lemma.
\end{proof}

Now we prove the main theorem of this article.

\begin{proof}[Proof of Theorem \ref{mainthm}]
 We start with estimating the norm $\|A_{r, \alpha}\|_{L^q(w)}^{q}$ by using the well know layer-cake formula
 \begin{equation*}
\|A_{r, \alpha}\|_{L^q(w)}^{q}=q\int_{0}^{\infty}\lambda^{q-1}w(A_{r, \alpha}f(x)>\lambda) \, d\lambda.
 \end{equation*}
 Then,
\begin{eqnarray*}
\|A_{r, \alpha}\|_{L^q(w)}^{q}&\leq& c_{\beta, q}\int_{0}^{\infty}\lambda^{q-1} \sum\limits_{n\in\mathbb{N}_{0}; 1\leq 2^n\leq k^r}2^{n q}\left(\frac{k^r}{2^n}\right)^{q \beta} \frac{k^{rq\epsilon(1-\alpha)}}{k^{rq}}w\left(\frac{|f| k^{r\alpha} }{\lambda}\geq 2^{n-1}\right)^{q/p}\, d\lambda\\
 &\leq &c_{\beta, q}\sum\limits_{n\in\mathbb{N}_{0}; 1\leq 2^n\leq k^r}2^{n q}\left(\frac{k^r}{2^n}\right)^{q \beta} \frac{k^{rq\epsilon(1-\alpha)}}{k^{rq}}\int_{0}^{\infty}\lambda^{q-1}w\left(\frac{|f| k^{r\alpha} }{\lambda}\geq 2^{n-1}\right)^{q/p}\, d\lambda\\
 &=&c_{\beta, q}\sum\limits_{n\in\mathbb{N}_{0}; 1\leq 2^n\leq k^r}2^{n q}\left(\frac{k^r}{2^n}\right)^{q \beta} \frac{k^{rq\epsilon(1-\alpha)}}{k^{rq}}\displaystyle\int_{0}^{\infty}\lambda^{q-1}\left(\sum_{x\in \mathcal{T}^k} \chi_{\{\frac{|f| k^{r\alpha} }{\lambda}\geq 2^{n-1}\}}(x) w(x)\right)^{q/p}\, d\lambda
\end{eqnarray*}
Using Minkowski's inequality the above is dominated by the following
\begin{align}
\nonumber & c_{\beta, q}\sum\limits_{n\in\mathbb{N}_{0}; 1\leq 2^n\leq k^r}2^{n q}\left(\frac{k^r}{2^n}\right)^{q \beta} \frac{k^{rq\epsilon(1-\alpha)}}{k^{rq}}\left(\sum_{x\in \mathcal{T}^k}\left(\int_{0}^{\infty} \chi_{\{x: \frac{|f(x)| k^{r\alpha} }{\lambda}\geq 2^{n-1}\}} \lambda^{q-1} \ d\lambda\right)^{p/q}w(x)\ dx \right)^{q/p}\\
\nonumber &\leq c_{\beta, q}\sum\limits_{n\in\mathbb{N}_{0}; 1\leq 2^n\leq k^r}2^{n q}\left(\frac{k^r}{2^n}\right)^{q \beta} \frac{k^{rq\epsilon(1-\alpha)}}{k^{rq}}\left(\sum_{x\in \mathcal{T}^k} \frac{|f(x)|^p k^{r\alpha p}}{2^{np}}w(x) \right)^{q/p}\\
\nonumber &\leq c_{\beta, q}\frac{1}{k^{rq(1-\beta-\epsilon  (1-\alpha)-\alpha)}}\sum_{n\geq 0}
\frac{1}{2^{n q \beta}}\left(\sum_{x\in \mathcal{T}^k} |f(x)|^p w(x) \right)^{q/p}\\
\nonumber &\leq c_{\beta, q}\frac{1}{k^{rq(1-\beta-\epsilon  (1-\alpha)-\alpha)}}\|f\|_{L^p(w)}^{q}.
\end{align}
Now 
\begin{align*}
\|\mathcal{S}_{\alpha}\|_{L^q(w)}&\leq \sum_{r=0}^{\infty}\|A_{r, \alpha}f\|_{L^q(w)}\\
&\leq c_{\beta, q}\sum_{r=0}^{\infty}\frac{1}{k^{r(1-\beta-\epsilon  (1-\alpha)-\alpha)}}\|f\|_{L^p(w)}\\
& \lesssim \|f\|_{L^p(w)},
\end{align*}provided we choose $\beta$ such that $(1-\beta-\epsilon(1-\alpha)-\alpha)>0$, i.e., $\beta<(1-\epsilon)(1-\alpha)$. This completes the theorem. 
\end{proof}
\section{Examples}
\label{examples}
This section is devoted in constructing examples of weights belonging to the class $\mathcal{Z}^{\epsilon, \alpha}_{p, q}$.\\ 
\begin{proof}[Proof of Theorem~\ref{Example-Thm}:]
Let $E$ and $F$ be any two subsets of $\mathcal{T}^{k}$. Define $E_{j}=E\cap \cTj $ and $F_{i}= F\cap \cTi$. Then
\begin{align}
\label{ex-est-1}
\sum_{x\in E}w(F\cap S(x, r))=\sum_{m=0}^{\infty}\, \sum_{\substack{i, j\in \mathbb{N}_{0}\\ i=j+r-2m}}\, \sum_{x\in E_{j}}w(F_{i}\cap S(x, r)),   
\end{align}
since for $x\in E_{j}$ and $y\in F_{i}$ we have $d(x, y)=r$ if and only if $i=j+r-2m$ for some $m\in \{0,\cdots, r\}$. The following estimate holds trivially since for each member of $\cTi$ there can be at most $k^m$ elements of $\cTj$ with distance $r$.
\begin{align}
\label{ex-est-2}    
\sum_{x\in E_{j}}w(F_{i}\cap S(x, r))\leq k^{m} w(F_{i}).
\end{align}
Now by our assumption \eqref{cond-weight-1},
\begin{equation}
\label{ex-est-3}
\sum_{x\in E_{j}}w(F_{i}\cap S(x, r))
\leq \sum_{x\in E_{j}}k^{\frac{r+i-j}{2}(p-\delta)} k^{r\delta} w(x)^{q/p}\leq k^{(r-m)(p-\delta)}k^{r\delta} w(E_{j})^{q/p}.
\end{equation}
Hence from \eqref{ex-est-2} and \eqref{ex-est-3} we obtain
\begin{align}
\label{ex-est-4}
\nonumber&\sum_{x\in E_{j}}w(F_{i}\cap S(x, r))\lesssim \min\Big\{k^{(r-m)(p-\delta)}k^{r\delta} w(E_{j})^{q/p}, k^{m} w(F_{i})\Big\}.
\end{align}
Therefore the rest of the proof is devoted to obtain an estimate for the following quantity:
$$M:=\sum_{m=0}^{\infty}\, \sum_{\substack{i, j\in \mathbb{N}_{0}\\ i=j+r-2m}}\min\Big\{k^{(r-m)(p-\delta)}k^{r\delta} w(E_{j})^{q/p}, k^{m} w(F_{i})\Big\}.$$
Denote $A_{j}=\frac{w(E_{j})^{q/p}}{k^{(p-\delta)j}}$ and $B_{j}=\frac{w(F_{j})}{k^{j}}$ for $j\geq 0$ and $A_{j}=B_{j}=0$ for $j<0$. Observe that
\begin{equation}
\label{ex-est-5}
\sum\limits_{j\geq 0}k^{(p-\delta)j} A_{j}\leq w(E)^{q/p}  ~~\text{and}~~  \sum\limits_{j\geq 0} k^{j} B_{j}=w(F).
\end{equation}

Note that for a real parameter $\rho>0$ to be chosen later we have
\begin{eqnarray}
\nonumber M
  & = & \sum_{m=0}^{\infty}\, \sum_{\substack{i, j\in \mathbb{N}_{0}\\ i=j+r-2m}}\min \Big\{k^{(r-m+j)(p-\delta)}k^{r\delta} A_{j}, k^{m+j}B_{i}\Big\}\\
 \nonumber & = & \sum_{m=0}^{\infty}\, \sum_{\substack{i, j\in \mathbb{N}_{0}\\ i=j+r-2m}}\min \Big\{k^{\frac{(i+j+r)(p-\delta)}{2}}k^{r\delta} A_{j}, k^{m+j}B_{i}\Big\}\\
 \nonumber & \leq & k^{\frac{p+\delta}{2}r}\sum_{j=0}^{\infty}\sum_{i<j+\rho} k^{\frac{(i+j)(p-\delta)}{2}} A_{j}+k^{r/2}\sum_{i=0}^{\infty} \sum_{j\leq i-\rho} k^{\frac{i+j}{2}} B_{i}\\
\nonumber  & \leq & k^{\frac{p+\delta}{2}r} k^{\frac{\rho(p-\delta)}{2}}\sum_{j=0}^{\infty} k^{j(p-\delta)}A_{j}+k^{r/2} k^{-\frac{\rho}{2}} \sum_{i=0}^{\infty} k^{i} B_{i}\\
& \leq & k^{\frac{p+\delta}{2}r} k^{\frac{\rho(p-\delta)}{2}} w(E)^{q/p}+k^{r/2} k^{-\frac{\rho}{2}} w(F),
\label{ex-est-6}
\end{eqnarray}
where in the last inequality we have used \eqref{ex-est-5}. We need to optimize the parameter $\rho$ in order to complete the proof. Let 
$$
\phi_{a, b}(\rho):=a \, k^{\frac{\rho(p-\delta)}{2}}+k^{-\frac{\rho}{2}}\, b
$$
for some positive constants $a$ and $b$. The function $\phi$ attains its absolute minimum at the point $\rho=\frac{2}{p+1-\delta}\log_{k}\left(\frac{b}{a(p-\delta)}\right)$.

Now choosing $a_{0}=k^{\frac{p+\delta}{2}r}w(E)^{q/p}$ and $b_{0}=k^{r/2}w(F)$, we obtain that \eqref{ex-est-6} attains its minimum value if we choose  $\rho=\rho_{0}:=\frac{2}{p+1-\delta}\log_{k}\left(\frac{k^{r/2}w(F)}{k^{\frac{p+\delta}{2}r}w(E)^{q/p}(p-\delta)}\right)$ and the value is given by the following

\begin{eqnarray*}
    \phi_{a_0, b_0}(\rho_0) & = &
    k^{\frac{p+\delta}{2}r}w(E)^\frac{q}{p} k^{\frac{p-\delta}{p+1-\delta}\log_{k}\left(\frac{k^\frac{r}{2}w(F)}{(p-\delta)k^{\frac{p+\delta}{2}r}w(E)^\frac{q}{p}}\right)}+k^\frac{r}{2}w(F) k^{\frac{-1}{p+1-\delta}\log_{k}\left(\frac{k^\frac{r}{2}w(F)}{(p-\delta)k^{\frac{p+\delta}{2}r}w(E)^\frac{q}{p}}\right)}\\
    &=& k^{\left(\frac{p+\delta}{2}-\frac{(p-\delta)(p+\delta-1)}{p+1-\delta}\right)r}\, w(E)^\frac{q}{p}\left(\frac{w(F)}{(p-\delta)w(E)^\frac{q}{p}}\right)^{\frac{p-\delta}{p+1-\delta}}\\
    &&\qquad+\ k^{\frac{r}{2}\left(1-\frac{1-p-\delta}{p+1-\delta}\right)}w(F)\left(\frac{(p-\delta)w(E)^\frac{q}{p}}{w(F)}\right)^{\frac{1}{p+1-\delta}}
\end{eqnarray*}

We conclude that 
\begin{equation}\label{final-est}
    \phi_{a_0, b_0}(\rho_0)\lesssim
    k^{\frac{p}{p-\delta+1}r} w(F)^{\frac{p-\delta}{p+1-\delta}} w(E)^{\frac{q}{p(p+1-\delta)}}.
\end{equation}

For the choice of the $\delta=\frac{1-\alpha p-\alpha p^2}{1-\alpha p}$ and $\frac{1}{q}=\frac{1}{p}-\alpha$, we have
\begin{align*}
 \frac{p-\delta}{p+1-\delta}&=\frac{p-\frac{1-\alpha p-\alpha p^2}{1-\alpha p}}{p+1-\frac{1-\alpha p-\alpha p^2}{1-\alpha p}} \\
 &=\frac{p-1+\alpha p}{p}\\
 &=\frac{1}{p'}+\alpha=1-\frac{1}{q}.
\end{align*}
Similarly $\frac{1}{p-\delta+1}=\frac{1-\alpha p}{p}=\frac{1}{q}$. Combining these facts with \eqref{final-est} we conclude
\begin{eqnarray*}
\sum_{x\in E}w(F\cap S(x, r))
&\leq & C_{p, \alpha} k^{\frac{p}{p-\delta+1}r} w(F)^{\frac{p-\delta}{p+1-\delta}} w(E)^{\frac{q}{p(p+1-\delta)}}\\
&=&C_{p, \alpha}  k^{(1-\alpha p)r} w(E)^{\frac{1}{p}} w(F)^{1-\frac{1}{q}}\\
&= &C_{p, \alpha} k^{(1-\alpha)\epsilon r} w(E)^{\frac{1}{p}} w(F)^{1-\frac{1}{q}},
\end{eqnarray*}
provided we choose $\epsilon=\frac{1-p\alpha}{1-\alpha}\in (0, 1)$. This completes the proof.
\end{proof}

\begin{proof}[Proof of Theorem~\ref{Thm-Ex}]
For $\beta=0$, the weight is the constant weight and the result trivially follows. Hence we assume $0< \beta\leq \frac{p(p-1)}{q}$. We have $1<p<\frac{1}{\alpha}$ and $\frac{1}{q}=\frac{1}{p}-\alpha$. $$w(x)=\sum_{j\geq 0} k^{\beta j}\chi_{\cTj}.$$
It is sufficient to show that $w$ satisfies condition~\eqref{cond-weight-1}. Let $x\in \cTj$ and $|i-j|\leq r$ and $m\in \{0,\cdots, r\}$ be the unique integer such that $i=j+r-2m$. Then
\begin{eqnarray*}
 w(\cTi\cap S(x, r))   
 & \leq & |\cTi\cap S(x, r)| k^{\beta i}\\
& \leq & k^{r-m}k^{(i-j)\frac{\beta q}{p}} k^{i(\beta-\frac{\beta q}{p})} k^{\frac{j\beta q}{p}}\\
 & = & k^{r-m}k^{(r-2m)\frac{\beta q}{p}} k^{i(\beta-\frac{\beta q}{p})} w(x)^{q/p}\\
 & = & k^{(r-m)(1+2 \frac{\beta q}{p})}k^{-r\frac{\beta q}{p}} k^{i(\beta-\frac{\beta q}{p})} w(x)^{q/p}.
\end{eqnarray*}
In order to complete the proof it is enough to show that the following inequality holds
\begin{equation*}
k^{(r-m)(1+2 \frac{\beta q}{p})}k^{-r\frac{\beta q}{p}} k^{i(\beta-\frac{\beta q}{p})}\leq k^{(r-m)(p-\delta)} k^{r\delta}.
\end{equation*}
But this is equivalent to
\begin{equation}\label{cor-1}
(r-m)\left(1+2 \frac{\beta q}{p}-p+\delta\right)\leq r\delta+ r\frac{\beta q}{p}-i(\beta-\frac{\beta q}{p}).
\end{equation}
Since $(\beta-\frac{\beta q}{p})\leq 0$ and we need to prove the above for all $i\geq 0$, thus \eqref{cor-1} will follow if we ensure the following
\begin{align*}
 r\left(1+2 \frac{\beta q}{p}-p+\delta\right)\leq r\delta+ r\frac{\beta q}{p}
 \iff \beta\leq \frac{p(p-1)}{q}.
\end{align*}
This completes the proof of Theorem~\ref{Thm-Ex}.
\end{proof}

\begin{remark}[Two weight estimates]
We would like to highlight that approach is also applicable in the two weight setting. Let $u, v$ be two weights on $\cT$ satisfying the following condition: Let $1<p\leq q<\infty$. There exists $\epsilon\in (0, 1)$ such that for all $E, F\subset \cT$ and $r\in \mathbb{N}_{0}$
\begin{align}
\label{two-weight-cond}
\sum_{x\in E}u(F\cap S(x, r))\leq C \ k^{\epsilon r(1-\alpha)} v(E)^{\frac{1}{p}}u(F)^{1-\frac{1}{q}}.
\end{align}
Then one can prove $\mathcal{S}_{\alpha}$ maps $L^p(v)$ to $L^q(u)$. We just point out that arguing as in Lemma~\ref{main-lemma}, under the condition \eqref{two-weight-cond} the following holds
$$u(\{A_{r, \alpha}f>\lambda\})\leq \frac{1}{(2^\beta-1)^q} \sum\limits_{n\in\mathbb{N}_{0}; 1\leq 2^n\leq k^r}2^{n q}\left(\frac{k^r}{2^n}\right)^{q \beta} \frac{k^{rq\epsilon(1-\alpha)}}{k^{rq}}v\left(\Big\{x\in \cT: \frac{|f| k^{r\alpha} }{\lambda}\geq 2^{n-1}\Big\}\right)^{q/p},$$
for all $\lambda>0,\  r\in \mathbb{N}_{0}$ and $\beta\in (0, 1)$. After this the proof follows the exact arguments as in Theorem~\ref{mainthm} with necessary modifications. It is also possible to provide an analogous condition of \eqref{cond-weight-1} in the two weight setting but for brevity we do not include it here.
\end{remark}

\section*{Acknowledgements}
The first author is supported by the institute postdoctoral fellowship from Tata Institute of Fundamental Research, Centre for Applicable Mathematics.

\bibliographystyle{amsalpha}

\end{document}